\documentclass{amsart}

\usepackage{epsfig}
\usepackage{amsthm,amsfonts}
\usepackage{amssymb,graphicx,color}
\usepackage[all]{xy}
\usepackage{verbatim}
\usepackage{hyperref}
\usepackage{enumitem}

\usepackage[square,comma]{natbib}
\usepackage[a4paper,top=2.5cm,bottom=2.5cm,left=3.5cm,right=3.5cm]{geometry}

\newtheorem{theorem}{Theorem}[section]
\newtheorem*{theorem*}{Theorem}

\newtheorem{corollary}[theorem]{Corollary}
\newtheorem{proposition}[theorem]{Proposition}
\newtheorem{definition}[theorem]{Definition}

\newtheorem{example}[theorem]{Example}
\newtheorem{remark}[theorem]{Remark}

\newcommand{\R}{\mathbb{R}}
\newcommand{\C}{\mathbb{C}}
\newcommand{\N}{\mathbb{N}}

\makeatletter
\newcommand{\mlabel}[2]{\def\@currentlabel{#2}\label{#1}} 
\makeatother

\begin{document}

\title[Tangent cones of Lip. normally embedded sets are Lip. normally embedded]
{Tangent cones of Lipschitz normally embedded sets are Lipschitz normally embedded. Appendix by Anne Pichon and Walter D Neumann}

\author{Alexandre Fernandes}
\author{J. Edson Sampaio}
\address{Alexandre ~Fernandes and J. Edson ~Sampaio -
Departamento de Matem\'atica, Universidade Federal do Cear\'a Av.
Humberto Monte, s/n Campus do Pici - Bloco 914, 60455-760
Fortaleza-CE, Brazil}
 \email{alex@mat.ufc.br}
 \email{edsonsampaio@mat.ufc.br}

\thanks{The first named author were partially supported by CNPq-Brazil}

\begin{abstract}
We prove that tangent cones of Lipschitz normally embedded sets are Lipschitz normally embedded. We also extend to real subanalytic sets  the notion of reduced tangent cone  and  we show that subanalytic Lipschitz normally embedded sets have reduced tangent cones. In particular, we get that Lipschitz normally embedded complex analytic sets have reduced tangent cones.
\end{abstract}

\maketitle

\section{Introduction}

This work was inspired in the following question which we learned from A. Pichon and W. Neumann in 2015: {\it does Lipschitz normally embedded complex analytic singularities have reduced and Lipschitz normally embedded tangent cones?}

A subset $X\subset\R^n$ is called \emph{Lipschitz normally embedded} when the inner and outer distance are bi-Lipschitz equivalent on $X$ (see Definition \ref{definition-ne}).  The notion of normal embedding, in the setting of Lipschitz Geometry of Singularities, has been  studied since 2000s  with the seminal paper \cite{BirbrairM:2000}. In \cite{NeumannPP}, the authors  showed that minimal singularities are characterized among rational surface singularities by being Lipschitz normally embedded. Another interesting work on this subject is the recent preprint \cite{KernerPR:2017} where it is proved normal embedding for several determinantal singularities. At the 2016 S\~ao Carlos Workshop on Real and Complex Singularities, A. Pichon announced a positive answer for the question above in the case of normal surfaces and also gave an example which adressed questions about a converse of this result (see Appendix). 

The main goals of this paper are to extend to real subanalytic sets the notion of  reduced tangent cones and to show that the same statement of Neumann-Pichon result above holds true for such a class of sets, that is, subanalytic Lipschitz normally embedded subsets $X\subset\R^n$ have reduced and Lipschitz normally embedded tangent cones.

In order to know details about subanalytic sets, see, for example, \cite{BierstoneM:1988}.

\bigskip

\section{Lipschitz normally embedded sets and tangent cones}\label{section:preliminaries}
\subsection{Lipschitz normally embedded sets}

Let $X\subset\R^m$ be a path connected subset. Let us consider the following distance on $X$:  given two points $x_1,x_2\in X$, $d_X(x_1,x_2)$  is the infimum of the lengths of paths on $X$ connecting $x_1$ to $x_2$. This is what we are going to call
{\bf inner distance}  on $X$. Let us observe that
$$ \| x_1 - x_2 \| \leq d_X(x_1,x_2), \quad \forall \ x_1,x_2\in X.$$

\begin{definition}\label{definition-ne}
A subset $X\subset\R^m$ is called \emph{{\bf Lipschitz normally embedded}} if there exists a constant $\lambda\geq 1$ such that
$$d_X(x_1 , x_2) \leq  \lambda \| x_1 - x_2 \|, \quad \forall \ x_1,x_2\in X.$$
Given $p\in X$, we say that $X$ is \emph{{\bf Lipschitz normally embedded at $p$}} if there exists an open neighborhood $U\subset \R^m$ of $p$ such that $X\cap U$ is Lipschitz normally embedded.
\end{definition}

It is valuable to observe that, if $p$ is a smooth point of a subset $X\subset\R^m$ in the sense that there is a neighborhood $V$ of $p$ in $X$ such that $V\subset\R^n$ is a smooth submanifold, then $X$ is Lipschitz normally embedded at $p$. In particular, any compact, connected submanifold of $\R^m$  is Lipschitz normally embedded. In other words, normal embedding property is, somehow, a measure of regularity of Euclidean subsets.

In another way, let us finish this subsection giving examples  of  Euclidean subsets which are not Lipschitz normally embedded, for instance, the real (or complex) cusp
$$\{(x,y) \in \mathbb{K}^2 \ | \ x^3=y^2 \}$$ is not Lipschitz normally embedded at the origin $0\in\mathbb{K}^2$ ($\mathbb{K}=\R,\C$).

Let us emphasize that all the Euclidean subsets considered in this paper are supposed to be equipped with the Euclidean induced metric.

\subsection{Tangent cones}

Let $X\subset \R^m$ be a subset and let $p\in \overline{X}$.
\begin{definition}
Fixed a sequence of positive real numbers $S=\{t_j\}_{j\in \N}$ such that $\lim\limits_{j\to \infty} t_j=0,$ we say that a vector $v\in \R^m$ is a {\bf direction of $X$ at $p$} (with respect to $S$), if there exist a sequence $\{x_j\}_{j\in \N} \subset X\setminus \{p\}$ and $j_0\in \N$ satisfying $\|x_j-p\|=t_j$ for all $j\geq j_0$ and $\lim\limits_{j\to \infty}\frac{x_j-p}{t_j}=v$. The set of all directions of $X$ at $p$ (with respect to $S$) is denoted by $D_p^{S}(X)$. 
\end{definition}
\begin{definition}
We say that $Z\subset \R^{m}$ is {\bf a tangent cone of $X$ at $p$} if there is a sequence $S=\{t_j\}_{j\in \N}$ of positive real numbers such that $\lim\limits_{j\to \infty} t_j=0$ and 
$$
Z=\{tv;\, v\in D_p^S(X)\mbox{ and }t\geq 0\}.
$$
When $X$ has a unique tangent cone at $p$, we denote it for $T_{p}X$ and we call $T_{p}X$ {\bf the tangent cone of $X$ at $p$}.
\end{definition}

Another way to present the tangent cone of a subset $X\subset\R^m$ at the origin $0\in\R^m$ is via the spherical blow-up of $\R^m$ at the point $0$ (it is also done in \cite{BirbrairFG:2017} and also in \cite{Sampaio:2017}) as it is going to be done in the following: let us consider the {\bf spherical blowing-up} at the origin of $\R^m$
$$
\begin{array}{ccl}
\rho_m\colon\mathbb{S}^{m-1}\times [0,+\infty )&\longrightarrow & \R^m\\
(x,r)&\longmapsto &rx
\end{array}
$$

Notice that $\rho_m\colon\mathbb{S}^{m-1}\times (0,+\infty )\to \R^m\setminus \{0\}$ is a homeomorphism with inverse mapping $\rho_m^{-1}\colon\R^m\setminus \{0\}\to \mathbb{S}^{m-1}\times (0,+\infty )$ given by $\rho_m^{-1}(x)=(\frac{x}{\|x\|},\|x\|)$. Let us denote
$$X':=\overline{\rho_m^{-1}(X\setminus \{0\})} \mbox{ and } \partial X':=X'\cap (\mathbb{S}^{m-1}\times \{0\}).$$

\begin{remark}
{\rm It is easy to see that $D_0^S X\times \{0\}\subset \partial X'$, for any sequence of positive real numbers $S=\{t_j\}_{j\in \N}$ such that $\lim\limits_{j\to \infty} t_j=0$. If $X\subset \R^m$ is a subanalytic set and $0\in X$, then $X$ has a unique tangent cone at $0$ and $\partial X'=\mathbb{S}_0X\times \{0\}$, where $\mathbb{S}_0X=T_0X\cap \mathbb{S}^{m-1}$.}
\end{remark}
Let us show why we required subanalycity of the set $X$ to get the claim above. If $w=(v,0)\in \partial X'$, then by Curve Selection Lemma, there is a continuous curve $\alpha: [0,\varepsilon)\to X$ such that $\lim\limits _{t\to 0^+}\frac{\alpha(t)}{\|\alpha(t)\|}=v$. Thus, if $S=\{t_j\}_{j\in \N}$ is a sequence of positive numbers such that $\lim\limits _{j\to +\infty } t_j=0$, then there is $j_0$ satisfying: $t_j\in[0,\delta ),$ for all $j\geq j_0$, where $\delta$ is the positive number such that $\|\alpha\|([0,\varepsilon ))=[0,\delta)$. Then, for each $j\geq j_0$, there is $s_j\in [0,\varepsilon )$ such that $\|\alpha(s_j)\|=t_j$. Thus, if we define $x_j=\alpha(s_j)$, when $j\geq j_0$ and $x_j=\alpha(s_{j_0})$, when $j< j_0$, we get
$$
\lim\limits _{j\to +\infty }\frac{x_j}{t_j}=v.
$$
This shows the other inclusion $\partial X'\subset D_0^S X\times \{0\}$.

\begin{remark}\label{remark-tangent-cone}
{\rm In the case where $X\subset \C^m$ is a complex analytic set such that $0\in X$, $T_{0}X$ is the zero set of a set of complex homogeneous polynomials (see \cite{Whitney:1972}, Theorem 4D). In particular, $T_{0}X$ is a complex algebraic subset of $\C^m$ given by the union of complex lines passing through the origin $0\in\C^m$. More precisely,  let $\mathcal{I}(X)$ be the ideal of $\mathcal{O}_m$ given by the germs which vanishes on $X$. For each $f\in\mathcal{O}_m$, let
$$ f = f_k+f_{k+1}+\cdots $$ be its Taylor development where each $f_j$ is a homogeneous polynomial of degree $j$ and $f_k\neq 0$. So, we say that $f_k$ is the \emph{initial part} of $f$ and we denote it by ${\bf in}(f)$. In this way, $T_0X$ is the affine variety of the ideal $\mathcal{I}^*(X)=\{{\bf in}(f);\, f\in \mathcal{I}(X)\}$}.
\end{remark}

\begin{theorem}\label{ne_set_ne_cone} Let $X$ be a subset of $\R^m$; $x_0\in X$. Suppose that $X$ has a unique tangent cone at $x_0$. If $X$ is Lipschitz normally embedded, then its tangent cone at $x_0$ is Lipschitz normally embedded as well.
\end{theorem}

\begin{proof} Let us suppose that $x_0$ is the origin of $\R^m$. Let $\lambda$ be a real positive number such that  $$d_X(x,y)\leq \lambda \|x-y\|, \quad \forall \ x,y\in X.$$
Given $\varepsilon >0$, let $K>\lambda+1+\varepsilon $. Given $x,y\in X$ such that $\|x\|=\|y\|=t$, we see that any arc $\alpha$ on X, connecting $x$ to $y$ such that $length(\alpha)\leq d_X(x,y)+\varepsilon \|x-y\|$, is contained in the compact Euclidean ball $B[0,Kt]$. In fact, according to the Figure \ref{curve}, any arc $\alpha$ connecting $x$ to $y$, which is not contained in the ball $B[0,Kt]$, has length at least $2(K-1)t$, in particular, $length(\alpha)>(\lambda+\varepsilon )\| x- y \|$.

\begin{center}\mlabel{curve}{1}
\includegraphics[height=8cm]{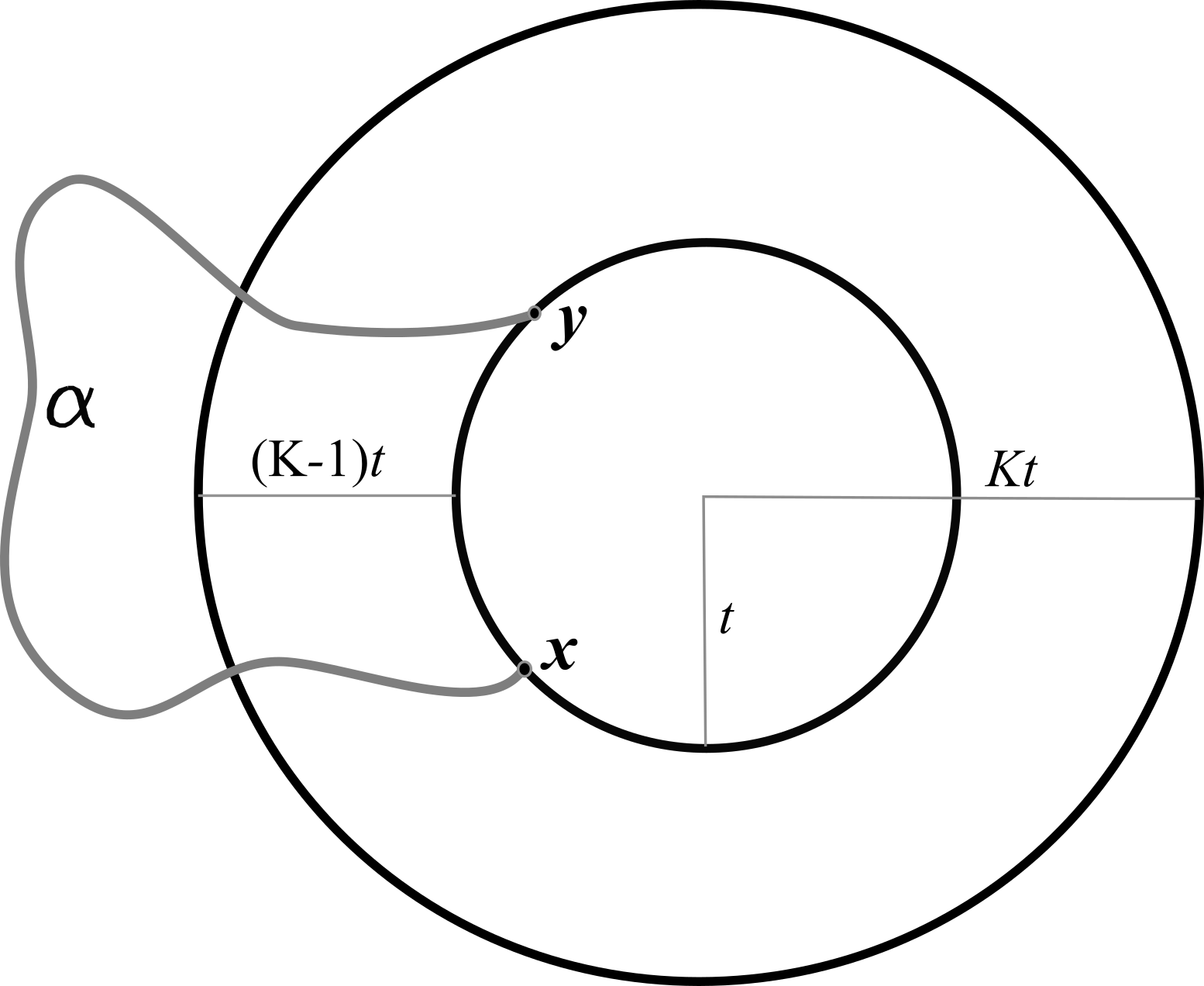}\\
{\bf Figure 1.}
\end{center}

Let $v,w\in T_0X$ be tangent vectors such that $\|v\|=\|w\|=1$. We claim that $d_{T_0X}(v,w)\leq \lambda \|v-w\|$. In fact, let $x_n$ and $y_n$ be sequences of points in $X$ such that $\|x_n\|=\|y_n\|=t_n\to 0$ and $$\frac{1}{t_n}x_n\to v \quad \quad  \mbox{and} \quad \quad \frac{1}{t_n}y_n\to w.$$ For each $n$, let $\gamma_n\colon[0,1]\rightarrow X$ be an arc on $X$ connecting $x_n$ to $y_n$ satisfying $length(\gamma_n)\leq d_X(x_n,y_n)+\varepsilon \|x_n-y_n\|$. Let us define $\alpha_n\colon[0,1]\rightarrow \R^m$ by $t_n\alpha_n=\gamma_n$. We have seen that $\alpha_n$ is contained in the compact set $B[0,kt_n]$, where $k$ is any constant greater than $\lambda+1+\varepsilon $. Moreover,
\begin{eqnarray*}
\ell_n:=length(\alpha_n)&=&\frac{1}{t_n}length(\gamma_n) \\
&\leq & \frac{1}{t_n}d_X(x_n,y_n) +\frac{\varepsilon}{t_n}\|x_n-y_n\|\\
&\leq &\frac{\lambda+\varepsilon}{t_n} \|x_n-y_n\|.
\end{eqnarray*}
Since $\displaystyle\frac{1}{t_n} \| x_n-y_n\| \to \| v-w\|$, we get that $\{\ell_n\}_{n\in \N}$ is bounded. Thus, by taking a subsequence if necessary, we can suppose that $\ell_n\leq \ell:=(\lambda+\varepsilon) \|v-w\|+1$ for all $n\in \N$. For each $n$, let $\widetilde \alpha_n:[0,\ell_n]\to X$ be the re-parametrization by arc length of $\alpha_n$ and let $\beta_n:[0,\ell]\to X$ be the curve given by $\beta_n(t)=\widetilde\alpha_n(t)$ if $t\in [0,\ell_n]$ and $\beta_n(t)=\widetilde \alpha_n(\ell_n)$ if $t\in [\ell_n, \ell]$. Since each $\widetilde \alpha_n$ is parametrized by arc length, for each $n\in \N$ we get $\|\beta_n(s)-\beta_n(t)\|\leq |s-t|$, for all $s,t\in[0,\ell]$. Then, $\{\beta_n\}$ is a equicontinuous and bounded family. Therefore, by using Arzel\`a-Ascoli Theorem (see, for example, Theorem 45.4 in \cite{Munkres:2014}), there exists a continuous arc $\alpha\colon [0,\ell]\rightarrow T_0X$ for which $\beta_n$ converges uniformly, up to subsequence; let us say $\beta_{n_k}\rightrightarrows\alpha$. 

Now, let us estimate length of $\alpha$. Since $\varepsilon >0$, there is a positive integer number $k_0$ such that, for any $k>k_0$, $$ length(\beta_{n_k})\leq (\lambda+\epsilon)\|v-w\|+\epsilon, $$ hence $length(\alpha)$ is at most $(\lambda+\epsilon)\| v-w\|+\epsilon$ and then $d_{T_0X}(v,w)\leq (\lambda+\epsilon)\|v-w\|+\epsilon.$ Since $\epsilon$ was arbitrarily chosen, we conclude that $d_{T_0X}(v,w)$ is at most $\lambda \| v-w\|$.

Finally, we are going to prove that for any pair of vectors $x,y\in T_0X$, their inner distance into the cone $T_0X$ is at most $(1+\lambda)\| x- y\|$. In fact, denote by $d_{T_0X}$ the inner distance of $T_0X$. If $x=sy$ for some $s\in \R$, then the line segment connecting $x$ and $y$ is a subset of ${T_0X}$. In this case, $d_{T_0X}(x,y)=\|x-y\| \leq (1+\lambda)\| x- y\|$. Hence, we can suppose that $x\not=sy$ for all $s\in \R$. In particular, $x\not =0$ and $y\not =0$. Moreover, if $y'=ty$, where $t=\frac{\|x\|}{\|y\|}$, we can see as in the Figure \ref{angle} below that the angle $\alpha$ is greater than $\frac{\pi}{2}$. Then, $\|x-y\|\geq \|x-y'\|$ and $\|x-y\|\geq \|y'-y\|$.
\begin{center}\mlabel{angle}{2}
\includegraphics[height=8cm]{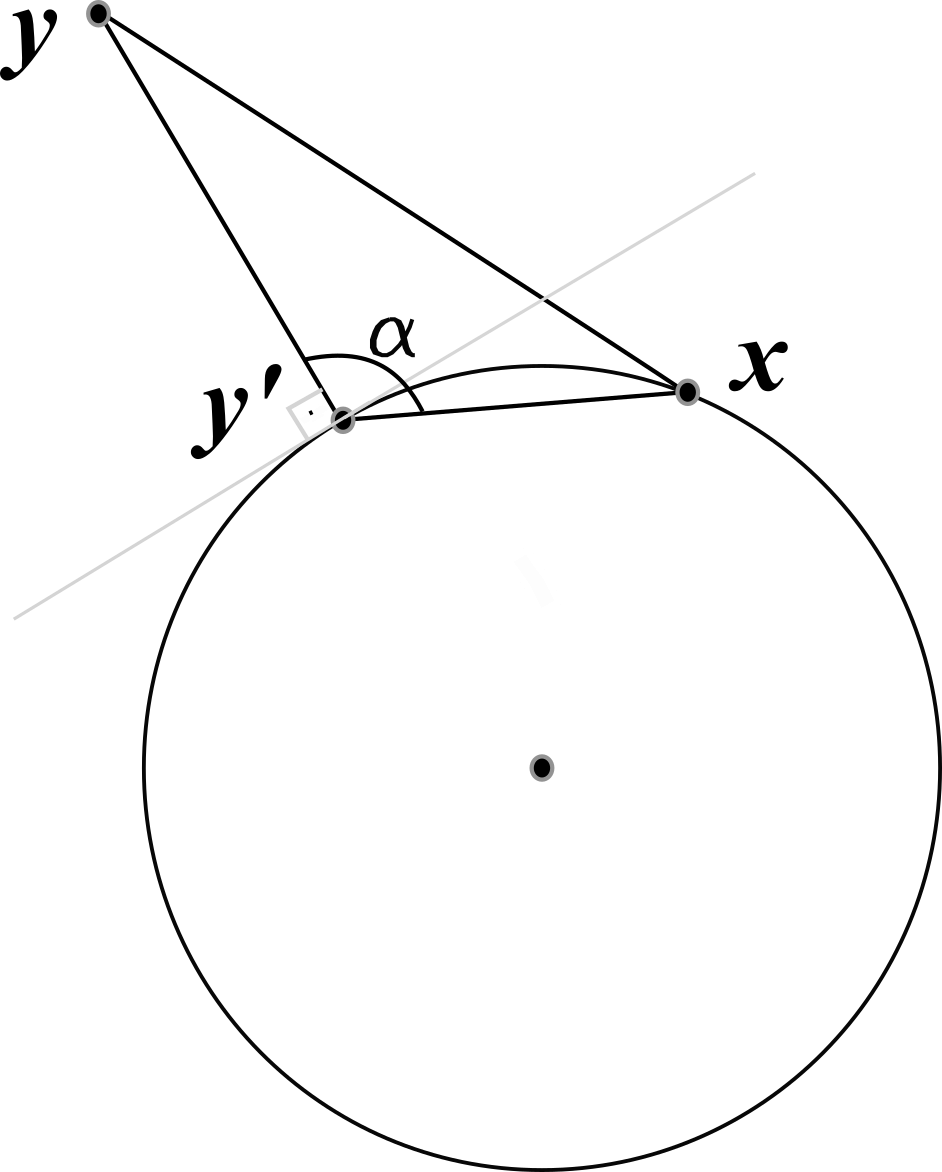}\\
{\bf Figure 2.}
\end{center}
Therefore,
\begin{eqnarray*}
(1+\lambda)\| x - y\|& = & \| x - y\| +\lambda \| x - y\| \\
&\geq& \| y'- y \| + \lambda \| x - y'\| \\
&\geq& d_{T_0X}(y,y') + d_{T_0X}(x,y') \\
&\geq& d_{T_0X}(x,y).
\end{eqnarray*}
This finishes the proof.
\end{proof}
According to the next example, the converse of the Theorem \ref{ne_set_ne_cone} does not hold true.
\begin{example}
 Let $X=\{(x,y)\in \R^2;\, y^2=x^3\}$. Then $T_0X=\{(x,0)\in\R^2;\, x\geq 0\}$ is a Lipschitz normally embedded set at $0$. However, $X$ is not a Lipschitz normally embedded set at $0$.
\end{example}

Next we have an example where we show that the set is not Lipschitz normally embedded at the origin $0\in\R^4$ by applying the Theorem \ref{ne_set_ne_cone}.

\begin{example}\label{2.8}
 Let $$X=\{(x,y,z,t)\in \R^4; \ ((x-t)^2+y^2+z^2-t^2)\cdot ((x+t)^2+y^2+z^2-t^2)=t^{10}, \ t\geq 0\}.$$
\begin{center}\mlabel{X_double_spheres}{3}
\includegraphics[height=8cm]{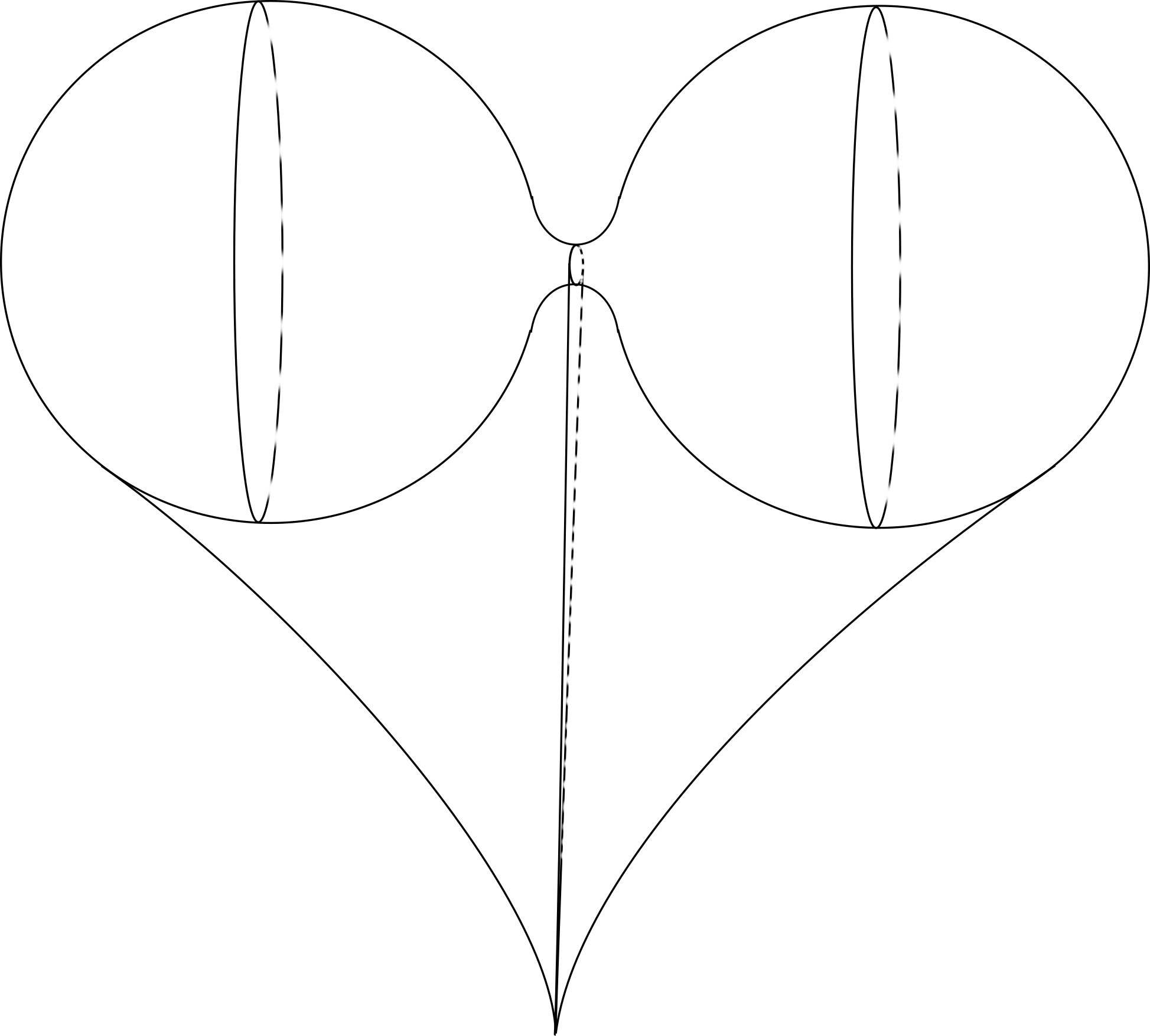}\\
{\bf Figure 3.}
\end{center}
In this case, $T_0X\subset\R^4$ is the cone over the following spheres
$$M=\{ (x,y,z,1)\in\R^4 \ (x-1)^2+y^2+z^2=1\} \cup \{ (x,y,z,1)\in\R^4 \ (x+1)^2+y^2+z^2=1\},$$ which is not Lipschitz normally embedded at $0$.
\begin{center}\mlabel{double_spheres}{4}
\includegraphics[height=7cm]{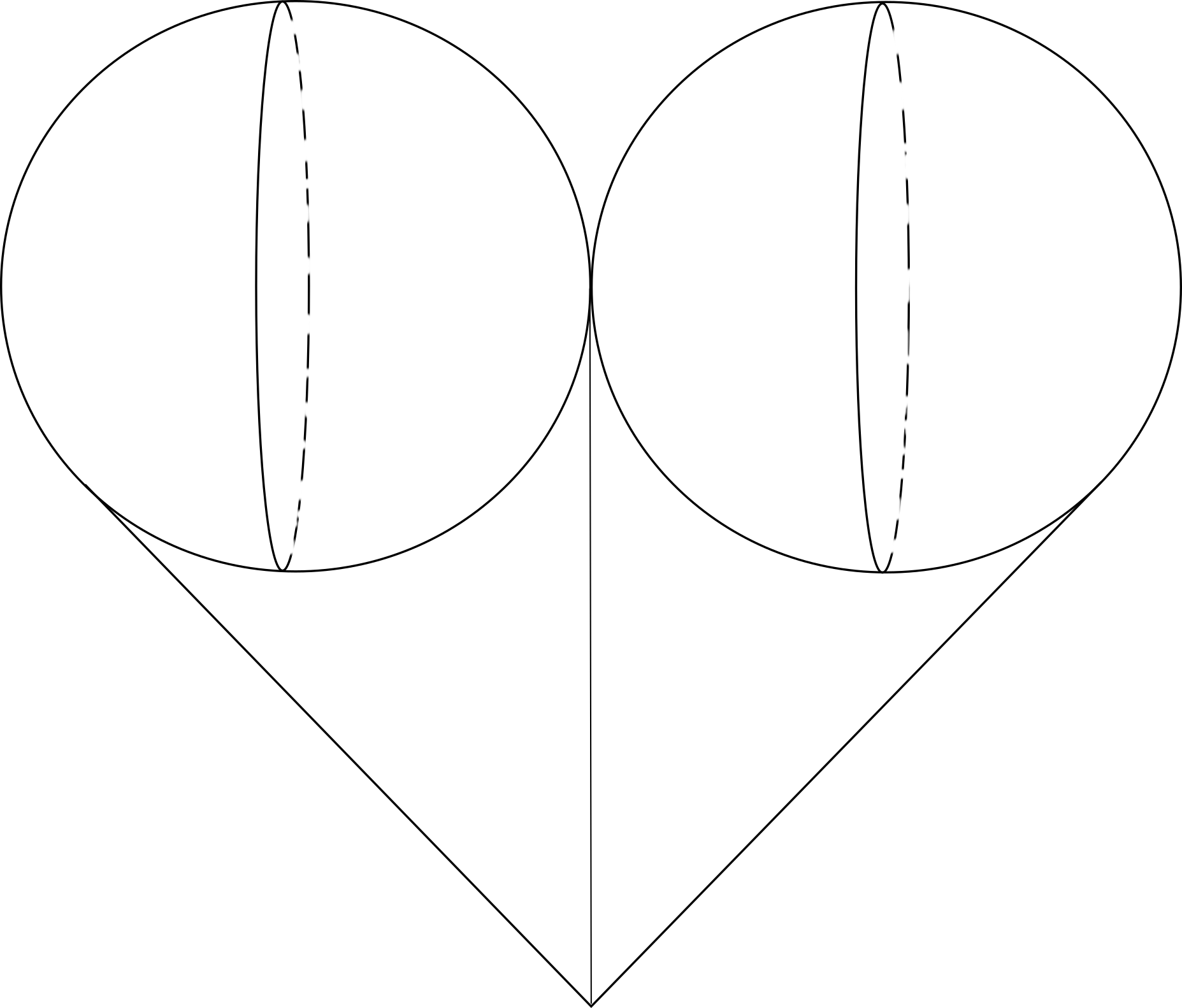}\\
{\bf Figure 4.}
\end{center}
\end{example}
In order to show that this cone is not Lipschitz normally embedded, by Proposition 2.8(b) in \cite{KernerPR:2017}, it is enough to show that $M$ is not Lipschitz normally embedded. Thus, if $M$ is Lipschitz normally embedded, then there is a constant $C\geq 1$ such that $d_M(x,y)\leq C\|x-y\|$ for all $x,y\in M$. Let $\alpha, \beta:(0,1)\to M$ be two curves given by $\alpha(s)=(s,(1-(1-s)^2)^{\frac{1}{2}},0,1)$ and $\beta(t)=(-s,(1-(1-s)^2)^{\frac{1}{2}},0,1)$. 
Then,
$$
\|\alpha(s)-\beta(s)\|=2s
$$
and
$$
d_X(\alpha(s),\beta(s))\geq 2(1-(1-s)^2)^{\frac{1}{2}}=2(2s-s^2)^{\frac{1}{2}}.
$$
Therefore, 
$$
C\geq (\frac{2}{s}-1)^{\frac{1}{2}}, \forall s\in (0,1)
$$
which is a contradiction, since the right hand side of this inequality is unbounded.

\section{On reduced tangent cones}

In this subsection, we closely follow definitions from the paper \cite{BirbrairFG:2017}.

\begin{definition}
Let $X\subset \R^m$ be a subanalytic subset such that $0\in X$. We say that $x\in\partial X'$ is {\bf simple point of $\partial X'$}, if there is an open $U\subset \R^{m+1}$ with $x\in U$ such that:
\begin{itemize}
\item [a)] the connected components of $(X'\cap U)\setminus \partial X'$, say $X_1,..., X_r$, are topological manifolds with $\dim X_i=\dim X$, for all $i=1,...,r$;
\item [b)] $(X_i\cup \partial X')\cap U$ is a topological manifold with boundary, for all $i=1,...,r$;.
\end{itemize}
Let $Smp(\partial X')$ denote the set of all simple points of $\partial X'$.
\end{definition}
\begin{remark}
{\rm By using Theorem 2.2 proved in \cite{Pawlucki:1985}, we get that $Smp(\partial X')$ is dense in $\mathbb{S}_0X\times \{0\}$ if $\dim \mathbb{S}_0X=\dim X-1$ and $X$ has pure dimension.}
\end{remark}
\begin{definition}
Let $X\subset \R^m$ be a subanalytic set such that $0\in X$.
We define $k_X:Smp(\partial X')\to \N$ by: $k_X(x)$ is the number of components of $\rho_m^{-1}(X\setminus\{0\})\cap U$, where $U$ is a sufficiently small open neighborhood of $x$.
\end{definition}

\begin{remark}
{\rm It is clear that the function $k_X$ is locally constant, hence $k_X$ is constant on each connected component $C_j$ of $Smp(\partial X')$. This is why we are allowed to set $k_X(C_j):=k_X(x)$ for any $x\in C_j$.}
\end{remark}
\begin{remark}
{\rm In the case that $X\subset \C^n$ is a complex analytic subset, let $X_1,...,X_r$ be the irreducible components of the tangent cone $T_0X$. Then there is a complex analytic subset $\sigma\subset \C^n$ with $\dim \sigma <\dim X$, such that for each $i=1,...,r$, $X_i\setminus \sigma$ intersect only one connected component $C_j$ (see \cite{Chirka:1989}, pp. 132-133). In this case, we define $k_X(X_i):=k_X(C_j)$.}
\end{remark}
\begin{remark}
{\rm The number $k_X(C_j)$ is equal to $n_j$ defined in \cite{Kurdyka:1989}, pp. 762, and also it is equal to $k_j$ defined in \cite{Chirka:1989}, pp. 132-133, in the case where $X$ is a complex analytic set.}
\end{remark}

Let $X\subset \C^m$ be a complex analytic subset with $0\in X$. Let $\mathcal{I}(X)$ be the ideal of $\mathcal{O}_m$ given by the germs which vanishes on $X$ and $\mathcal{I}^*(X)$ the ideal of the initial parts of $\mathcal{I}(X)$ (see Remark \ref{remark-tangent-cone}). We say that $X$ has {\bf reduced tangent cone} at $0$, if the ring $\mathcal{O}_m/\mathcal{I}^*(X)$ is reduced in the sense that it does not have nonzero nilpotent elements.

 \begin{remark}
 {\rm Notice that the notion of reduced tangent cone defined above is not an absolute concept. Indeed, the complex line
$$\{(x,y)\in\C^2 \ | \ y=0\}$$ is the tangent cone at the origin of the cusp and the parabola below
$$\{(x,y)\in\C^2 \ | \ y^2=x^3\} \ \mbox{and} \  \{(x,y)\in\C^2 \ | \ y=x^2\}.$$ 
This line as the tangent cone of the parabola at the origin is reduced but as the tangent cone of the cusp is not.}
 \end{remark}

\begin{remark}
{\rm Let $X\subset \C^m$ be a complex analytic subset with $0\in X$ and let $X_1,...,X_r$ be the irreducible components of $T_0X$. Using that the multiplicity of $X$ at $0$ is equal to the density of $X$ at $0$ (see the Theorem 7.3 in \cite{Draper:1969}) jointly with Th\'eor\`em 3.8 in \cite{Kurdyka:1989}, we get
$$
m(X,0)=\sum \limits_{i=1}^rk_X(X_i)m(X_i,0).
$$
Thus, $X$ has reduced tangent cone if and only if $k_X(X_i)=1$, $i=1,...,r$ (see Appendix E in \cite{Gau-Lipman:1983}).}
\end{remark}

The last remark allow us to state the following extension of the notion of reduced tangent cone for subanalytic sets.
\begin{definition}
 Let $X$ be a subanalytic subset in $\R^m$; $0\in X$.  We say that  $X$ has {\bf reduced tangent cone} at $0$, if $k_X(x)=1$ for all $x\in Smp(\partial X')$.
\end{definition}

\begin{theorem}
\label{multiplicities}
 Let $X$ be a subanalytic subset in $\R^m$; $0\in X$. If $X$ is a Lipschitz normally embedded set at $0$, then $X$ has reduced tangent cone at the origin.
\end{theorem}
\begin{proof}

Suppose that there is $x=(x',0)\in Smp(\partial X')=\mathbb{S}_0X\times \{0\}$ such that $k_X(x)\geq 2$. Then, there are $\delta,\varepsilon>0$ such that $k_X(y)=k_X(x)=k$ for all $y\in \partial X'\cap B_{\varepsilon}(x)$ and $X'\cap U_{\delta,\varepsilon}\setminus \partial X'$ has $k$ connected components, where
$$
U_{\delta,\varepsilon}=\{(w,s)\in \mathbb{S}^{m-1}\times [0,+\infty); \|w-x'\|<\varepsilon\mbox{ and }0\leq s<\delta\}.
$$
Let $X_1$ and $X_2$ be two connected components of $X'\cap U_{\delta,\varepsilon}\setminus \partial X'$ and $C_{\delta,\varepsilon}=\{v\in\R^m\setminus\{0\};\, \|v-sx'\|<s\varepsilon \mbox{ and }0<s<\delta\}$.

For each $n\in \N$ we can take $x_n\in \rho(X_1)$ and $y_n\in \rho(X_2)$ such that $\|x_n\|=\|y_n\|=t_n$ and $\lim \frac{x_n}{t_n}=\lim\frac{y_n}{t_n}=x'$. Moreover, taking subsequence, if necessary, we can assume that $x_n, y_n\in C_{\delta,\frac{\varepsilon}{2}}$. Thus, if $\beta_n:[0,1]\to X$ is a curve connecting $x_n$ to $y_n$, we have that there exists a $t_0\in [0,1]$ such that $\beta_n(t_0)\not \in C_{\delta,\varepsilon}$, since $x_n$ and $y_n$ belong to different connected components of $X\cap C_{\delta,\varepsilon}$. Therefore, $length(\beta_n)\geq \varepsilon t_n$ and, then, $d_X(x_n,y_n)\geq \varepsilon t_n$.

On the other hand, $X$ is a Lipschitz normally embedded set, then there is a constant $C>0$ such that $d_X(v,w)\leq C\|v-w\|$, for all $v,w\in X$. Therefore, $C\|\frac{x_n}{t_n}-\frac{y_n}{t_n}\|\geq \varepsilon,$ for all $n\in \N$
and this is a contradiction, since $\lim\frac{x_n}{t_n}=\lim\frac{y_n}{t_n}$.
\end{proof}
For subanalytic sets, we can sum up the Theorems \ref{ne_set_ne_cone} and \ref{multiplicities} as the following
\begin{corollary}\label{red-nn}
 Let $X\subset\R^m$ be a subanalytic set; $0\in X$. If $X$ is a Lipschitz normally embedded set at $0$, then $X$ has reduced tangent cone at $0$ and $T_0X$ is Lipschitz normally embedded.
\end{corollary}
As we can see in the next example, the converse of Corollary \ref{red-nn} is not true in general.
\begin{example}\label{3.12}
 Let $$X=\{(x,y,z)\in \R^3; \ (x^2+y^2-z^2)\cdot (x^2+(|y|-z-z^3)^2-z^6)=0, \ z\geq 0\}.$$
\begin{center}\mlabel{X_nne}{5}
\includegraphics[height=6cm]{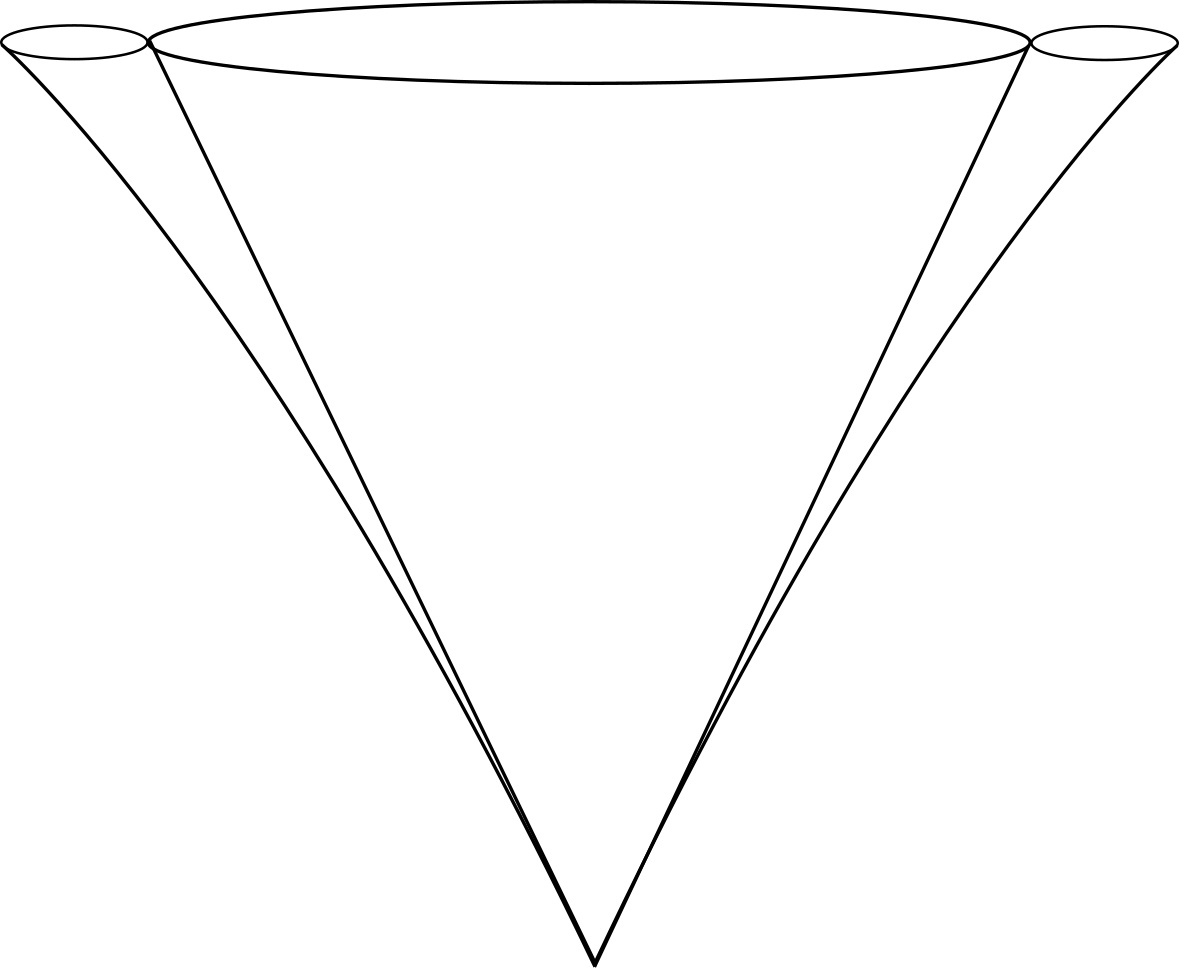}\\
{\bf Figure 5.}
\end{center}
In this case, $T_0X\subset\R^3$ is the cone
$$\{ (x,y,z)\in\R^3 \ x^2+y^2=z^2,\,\, z\geq 0\}$$
which is Lipschitz normally embedded at $0$ and $X$ has reduced tangent cone at $0$.
\begin{center}\mlabel{TX_rne}{6}
\includegraphics[height=5cm]{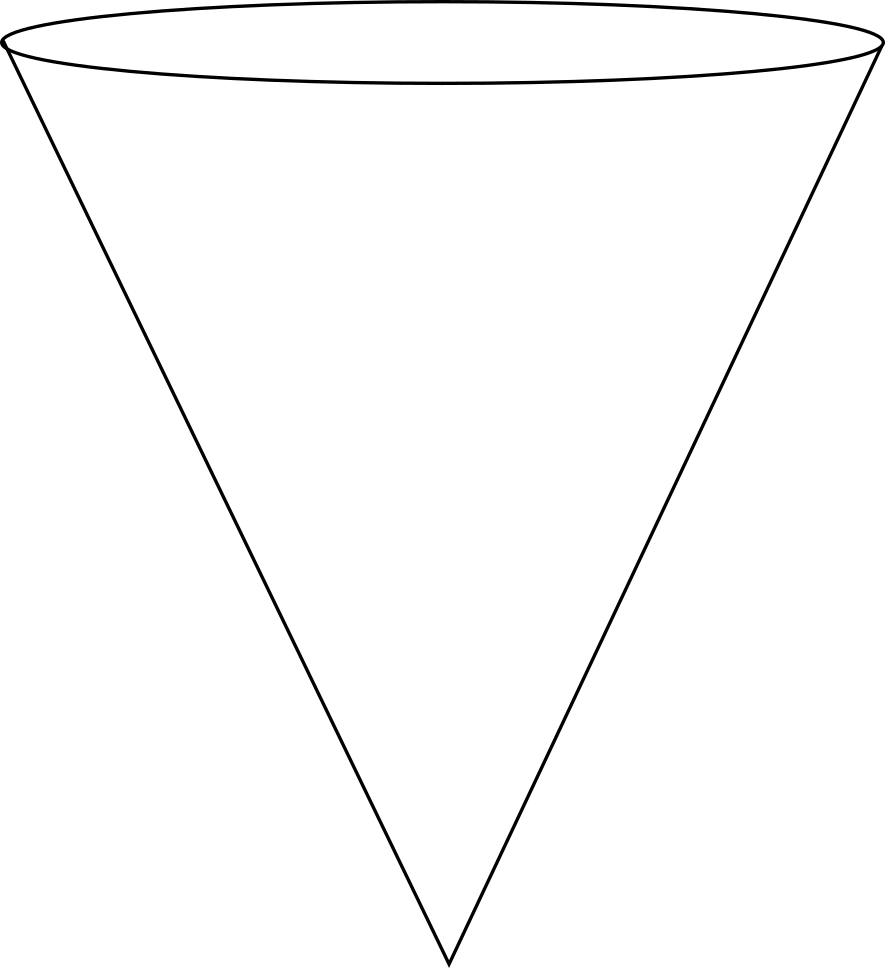}\\
{\bf Figure 6.}
\end{center}
\end{example}

Suppose that there is a constant $C\geq 1$ such that $d_X(x,y)\leq C\|x-y\|$ for all $x,y\in X$. Thus, for $r\in (0,1)$ fixed, let $\alpha, \beta:(0,r^3)\to X$ be two curves given by $\alpha(t)=(t,(r^2-t^2)^{\frac{1}{2}},r)$ and $\beta(t)=(t,r+r^3-(r^6-t^2)^{\frac{1}{2}},r)$. 
Then,
$$
\|\alpha(t)-\beta(t)\|=r+r^3-(r^6-t^2)^{\frac{1}{2}}-(r^2-t^2)^{\frac{1}{2}}
$$
and
$$
d_X(\alpha(t),\beta(t))\geq t.
$$
Therefore, 
$$
1\leq C\frac{r^3-(r^6-t^2)^{\frac{1}{2}}}{t}+C\frac{r-(r^2-t^2)^{\frac{1}{2}}}{t}, \forall t\in (0,r^3),
$$
which is a contradiction, since the right side of this inequality tends to zero when $t\to 0^+$.

As a consequence of Corollary \ref{red-nn}, we get a positive answer to the question of A. Pichon and W. Neumann quoted in the introduction of this work.
\begin{corollary}
 Let $X\subset\C^m$ be a complex analytic set; $0\in X$. If $X$ is a Lipschitz normally embedded set at $0$, then $X$ has reduced tangent cone at $0$ and $T_0X$ is Lipschitz normally embedded.
\end{corollary}

In the following, we give one example of complex surface $X \in \C^3$, with non-isolated singularity at $0$, such that $X$ has reduced tangent cone at $0$, its  tangent cone is normally embedded while X is not normally embedded at $0$. It shows that the converse statement of the above corollary does not hold true.

\begin{example}\label{3.14}
 Let $$X=\{(x,y,z)\in \C^3; \ y\cdot (x^2+(y-z^2)^2-z^4)=0\}.$$
In this case, $T_0X\subset\R^3$ is the cone
$$\{ (x,y,z)\in\C^3 \ y\cdot(x^2+y^2)=0\}$$
which is Lipschitz normally embedded at $0$ and $X$ has reduced tangent cone at $0$.
\end{example}

Suppose that there is a constant $C\geq 1$ such that $d_X(x,y)\leq C\|x-y\|$ for all $x,y\in X$. Thus, for $r\in (0,1)$ fixed, let $\alpha, \beta:(0,r^2)\to X$ be two curves given by $\alpha(t)=(t,0,r)$ and $\beta(t)=(t,r^2-(r^4-t^2)^{\frac{1}{2}},r)$. 
Then,
$$
\|\alpha(t)-\beta(t)\|=r^2-(r^4-t^2)^{\frac{1}{2}}
$$
and
$$
d_X(\alpha(t),\beta(t))\geq t.
$$
Therefore, 
$$
1\leq C\frac{r^2-(r^4-t^2)^{\frac{1}{2}}}{t}, \forall t\in (0,r^2),
$$
which is a contraction, since the right side of the above inequality tends to zero when $t\to 0^+$.

\begin{remark}
{\rm  The converse statement of the corollary above admits counterexamples even with isolated singularities. A. Pichon, in her talk at the S\~ao-Carlos Workshop on Real and Complex Singularities in 2016, gave one. The mentioned example by A. Pichon is the following: consider $X \in \C^3$ given by  equation $y^4+z^4+x^2(y+2z)(z+2y)^2+(x+y+z)^{11} = 0.$ This complex surface has reduced tangent cone at $0$, its  tangent cone is Lipschitz normally embedded while X is not Lipschitz normally embedded at $0$. See the Appendix to this paper.}
\end{remark}

Finally, we would like to observe that the results above hold true if we assume that the set $X$ is definable in a polynomial bounded structure on $\R$, instead  of the assumption that $X$ is a subanalytic set. In order to know details about definable sets in polynomially bounded o-minimal structures on $\R$, see, for example, \cite{Dries:1998}.

\noindent{\bf Acknowledgements}.  We are thankful to Anne Pichon and Walter Neumann for valuable discussions on this work. We would like to thank Lev Birbrair to his interest and some conversations about this subject as well. Finally, we would like to thank the anonymous referee for an accurate reading of the paper and essential suggestions which have improved the presentation of this paper.

\section{Appendix: A non normally embedded complex normal surface with normally embedded reduced tangent cone. By Anne Pichon and Walter D Neumann}

In this paper, A. Fernandes and E. Sampaio prove (Corollary
\ref{red-nn}) that if $(X,0) \subset (\mathbb R^m,0)$ is a subanalytic germ which is
Lipschitz normally embedded, then the two following conditions
hold:
\begin{itemize}
\item[(1)]the tangent cone $T_0X$ is reduced;
\item[(2)]$T_0X$ is Lipschitz normally embedded.
\end{itemize}

The following proposition gives an extra necessary condition $(3)$ for
normal embedding which is automatically satisfied when $(X,0)$ has an
isolated singularity.

Fix $\epsilon>0$ and consider a Lipschitz stratification  
$X \cap B_{\epsilon} = \coprod_{i \in A} S_{i}$    (see \cite{Mostowski,P}), so
$A$ is finite and for all $i \in A$ and all $y \in S_{i}$, the
bilipschitz type of $(X,y)$ does not depend on the choice of $y$ in
$S_{i}$. If $\epsilon >0$ is sufficiently small, one can assume that
$0 \in \overline{S_i}$ for all $i \in A$. We speak of a Lipschitz
stratification of the germ $(X,0)$.
 
\begin{proposition}\label{prop:extra condition} Let
  $(X,0) \subset (\mathbb R^m,0)$ be a subanalytic germ which is
  normally embedded. We choose a Lipschitz stratification of $(X,0)$
  as above. Then
\begin{itemize}
\item[(3)]for each stratum $S_i$ and $y_i \in S_i$, the germ $(X,y_i)$
  is Lipschitz normally embedded.
\end{itemize}
\end{proposition}

\begin{remark} 
  If $(X,0)$ has an isolated singularity at the origin, then a
  Lipschitz stratification of $(X,0)$ consists of the two strata
  $\{0\}$ and $X_{reg}$, so Condition (3) is void.

  Since $A$ is finite, Proposition \ref{prop:extra condition} implies
  that local Lipschitz normal embedding is an open property, i.e., if
  $(X,0)$ is  Lipschitz normally embedded, then there exists $\epsilon>0$ such
  that for all $y \in B_{\epsilon}$, $(X,y)$ is  Lipschitz normally embedded
  (this can also be proved using the Curve Selection Lemma).
\end{remark}

Example \ref{3.12} in the paper does not satisfy Condition (3).  In
this appendix we give an example which satisfies all three conditions
but which is not Lipschitz normally embedded. The example is a complex
germ which satisfies conditions (1) and (2) and has isolated
singularity, so Condition (3) is vacuously true.

\begin{proposition}\label{Example} The surface germ
  $(X,0) \subset (\mathbb C^3,0)$ with equation
  $$y^4 + z^4 + x^2(y+2z)(y+3z)^2 + (x+y+z)^{11} =0$$ is not Lipschitz
  normally embedded while it has isolated singularity at $0$ and its
  tangent cone is reduced and  Lipschitz normally embedded.
\end{proposition}

\begin{proof}[Proof of Proposition \ref{prop:extra condition}]
  Assume that there is a stratum $(S_i,0)$ along which the bilipschitz
  type of $X$ is not normally embedded. If $0 \in S_i$, then obviously
  $(X,0)$ is not normally embedded. Otherwise, $S_i \neq \{0\}$ and
  one can consider a sequence $(a_n)_{n\in \mathbb N}$ in $S_i$ converging to
  0. For each $n\in \mathbb N$, $(X,a_n)$ is not normally embedded, so one
  can choose two sequences of points $(a_{n,k})_{k \in \mathbb N^*}$ and
  $(b_{n,k})_{k \in \mathbb N^*}$ on $X$ converging to $a_n$ such that
  $d_o(a_n, a_{n,k}) <1/k$, $d_o(a_n, b_{n,k}) <1/k$ and
  $\frac{d_i(a_{n,k}, b_{n,k})}{d_o(a_{n,k}, b_{n,k})}\to\infty$ as
  $k\to\infty$ ($d_i$ and $d_o$ mean inner and outer metrics in
  $X$). Then the diagonal sequences $(a_{n,n})$ and $(b_{n,n})$ have
  $\lim_{n \to \infty} ||a_{n,n}|| = \lim_{n \to \infty} ||b_{n,n}|| =
  0$ while
  $\frac{d_i(a_{n,n}, b_{n,n})}{d_o(a_{n,n}, b_{n,n})}\to \infty$ as
  $n\to\infty$. So $(X,0)$ is not Lipschitz normally embedded.
\end{proof}
       
\begin{proof} [Proof of Proposition \ref{Example}]
  A simple calculation shows that $(X,0)$ has isolated singularity at
  $0$. Moreover, the tangent cone $T_0X$ is reduced and Lipschitz
  normally embedded since it is a union of planes.

  We will prove that $(X,0)$ is not Lipschitz normally embedded by
  showing the existence of two continuous arcs
  $p_1 \colon [0,1) \to X$ and $p_2\colon [0,1) \to X$ such that
  $p_1(0)=p_2(0)=0$ and
  $$\lim_{w \to 0} \frac{d_i(p_1(w),p_2(w))}{d_o(p_1(w),p_2(w))} =
  \infty.$$
 
  Let $\ell \colon (X,0) \to (\mathbb C^2, 0)$ be the generic projection
  defined as the restriction of the linear map $(x,y,z) \to
  (x,y)$. Consider the real arc $p$ in $(\mathbb C^2,0)$ parametrized by
  $x=w^2$, $y= w^5$ and $w \in [0,1)$. It is a radial arc inside the
  complex plane curve $\gamma$ given by $y=x^{5/2}$. We will use this
  curve $\gamma$ as a ``test curve'' for the geometry, along with a
  component $\gamma'$, described below, of $\ell^{-1}(\gamma)$.

  Computing, e.g., with Maple, one finds that the lifting $\ell^{-1}(p)$ of
  $p$ by $\ell$ consists of four arcs with two of them, $p_1(w)$ and
  $p_2(w)$, such that $d_o(p_1(w), p_2(w))$ is of order $w^{11/2}$.
  The arcs $p_1$ and $p_2$ are radial arcs inside a complex curve
  $\gamma'$ with Puiseux parametrization of the form
  $z=-\frac{1}{3} w^5 + a w^{11/2}+...$  (where $+...$
    means higher order terms). 
We will now prove that inner distance $d_i(p_1(w), p_2(w))$ is of order $w^q$ with $q \leq 5$. Since $q<11/2$, this gives the result. 

Let $\Pi$ be the polar curve of the projection $\ell$ and let
$\Delta = \ell(\Pi)$ be its discriminant curve.  Using Maple again one
shows that $\Delta$ has nine branches as follows:
\begin{itemize}
\item Five  branches with Puiseux expansions of the form
  $y=b x^2 + ...$;
\item Three  branches with Puiseux expansions of the form
  $y=b x^3 + ... $;
\item One branch with Puiseux expansions of the form
  $y=x^{11/4} + ... $\,.
\end{itemize}

Let $A_0$ be a $\Delta$-wedge around $\Delta$ as defined in  \cite[Proposition 3.4]{BNP}   and let $A'_0$ be the polar-wedge around $\Pi$ defined as the union of the components of  $\ell^{-1}(A_0)$ which contain components of $\Pi$ (see
\cite[Section 3]{BNP} for details). There are three cases: 
 
\noindent {\it Case 1: The curve $\gamma$ is outside $A_0$.} One can
refine the carrousel decomposition of $(\mathbb C^2,0)$ associated with $A_0$
(see \cite[Section 12]{BNP}) to obtain a carrousel decomposition which
has a $B(5/2)$-piece $B$ consisting of complex curves
$y= \lambda x^{5/2}$ with $\alpha \leq |\lambda| \leq \beta$ and
$ \alpha \leq 1 \leq \beta$. So $\gamma$ is inside $B$. Since $\ell$
is an inner bilipschitz homeomorphism outside $A_0$
(\cite[Proposition 3.4]{BNP}), then the components of $\ell^{-1}(B)$ are
$B(5/2)$-pieces. Then any real arc on $X$ from $p_1(w)$ to $p_2(w)$
has to travel through the $B(5/2)$-piece of $\ell^{-1}(B)$ containing
$\gamma'$.  Since distance to the origin inside $B$ is of order
$x=w^2$, this proves that $d_i(p_1(w^2), p_2(w^2))$ is of order
$x^{q'}$ with $q' \leq 5/2$, i.e., $d_i(p_1(w), p_2(w))$ is of order
$w^q$ with $q=2q' \leq 5$.

\noindent {\it Case 2: The curve $\gamma$ is inside $A_0$ and the
  curve $\gamma'$ is inside $A'_0$.} Then any geodesic arc on $X$ from
$p_1(w)$ to $p_2(w)$ has to cross $\Pi$, so
$$d_i(p_1(w^2), p_2(w^2)) \geq d_i(p_1(w^2), \Pi) + d_i(p_2(w^2), \Pi)
\geq d_o(p_1(w^2), \Pi) + d_o(p_2(w^2), \Pi).$$ Since $\ell$ is the
restriction of a linear projection, we have
$d_o(x,y) \geq d_o(\ell(x),\ell(y))$ for all $x,y \in X$. So we get:
$$d_i(p_1(w^2), p_2(w^2)) \geq 2 d_o(p(w^2), \Delta).$$
$d_o(p(w^2), \Delta) = x^{s}$ where $s=2q''$ and $q''$ is the contact
coefficient between the complex curves $\Delta$ and $\gamma$. The nine branches of $\Delta$ described earlier
show that   $q'' \leq 5/2$. Therefore  $d_i(p_1(w), p_2(w))$ is of order $w^q$ with $q  \leq 5$. 
 
 \noindent {\it Case 3: The curve $\gamma$ is inside $A_0$ and the
   curve $\gamma'$ is outside $A'_0$.} Since the curve $\gamma$ has
 contact exponent $5/2$ with any component of $\Delta$, then the polar
 rate $s'$ of $A_0$ (\cite[Proposition 3.4]{BNP}) is $s'\geq 5/2$,
 i.e., $A_0$ is a $D(s')$-piece with $s' \leq 5/2$.  Since $\ell$ is
 an inner bilipschitz homeomorphism outside $A'_0$, then the component
 of $\ell^{-1}(B)$ containing $\gamma'$ is a $D(s')$-piece and the
 same argument as in case 1 gives that $d_i(p_1(w), p_2(w))$ is of
 order $w^q$ with $q=2q' \leq 5$.
\end{proof}

\end{document}